\documentclass[12pt,a4paper,reqno]{amsart}
\usepackage{lineno}
\allowdisplaybreaks
\usepackage{amsmath}
\usepackage{amssymb}
\usepackage{amsfonts}
\usepackage{amsthm,latexsym,enumerate,url,cases}
\numberwithin{equation}{section}
\usepackage{mathrsfs}
\usepackage{array,longtable}

\usepackage{xcolor}
\definecolor{cite}{rgb}{0.00,0.60,1.00}
\definecolor{url}{rgb}{1.00,0.10,0.80}
\definecolor{link}{rgb}{0.00,0.00,1.00}
\usepackage[colorlinks,linkcolor=link,urlcolor=url,citecolor=cite,pagebackref,breaklinks]{hyperref}

\hypersetup{
	pdfstartpage=1,
	pdfstartview=FitH}
\def\Z{{\mathbb Z}}

\newcommand{\card}[1]{\left|#1\right|}
\newcommand{\floor}[1]{\left\lfloor #1\right\rfloor}

\theoremstyle{plain}
\newtheorem{theorem}{Theorem}

\newtheorem{lemma}{Lemma}
\newtheorem{problem}{Problem}
\newtheorem{corollary}{Corollary}
\newtheorem{proposition}{Proposition}
\newtheorem{conjecture}{Conjecture}
\theoremstyle{definition}

\newtheorem{remark}{Remark}

\usepackage{etoolbox}
\makeatletter
\patchcmd{\@settitle}{\uppercasenonmath\@title}{}{}{}
\patchcmd{\@setauthors}{\MakeUppercase}{}{}{}
\patchcmd{\section}{\scshape}{}{}{}
\makeatother

\begin{document}

\title
[Adjacent comparison bounds and extremal sets for Ruzsa numbers]
{Adjacent comparison bounds and extremal sets for Ruzsa numbers}

\author
[Y. Ding, H. Li, J. Li, N. Wei and X. Zhao] {Yuchen Ding, Huixi Li, Junfeng Li, Wei Niu and Xiamiao Zhao}

\address{(Yuchen Ding$^{1,2}$) $^1$School of Mathematics,  Yangzhou University, Yangzhou 225002, People's Republic of China}
\address{$^2$HUN-REN Alfr\'ed R\'enyi Institute of Mathematics, Budapest, Pf. 127, H-1364 Hungary}
\email{ycding@yzu.edu.cn}

\address{(Huixi Li) School of Mathematical Sciences and LPMC, Nankai University, Tianjin
300071, People's Republic of China}
\email{lihuixi@nankai.edu.cn}

\address{(Junfeng Li) School of Mathematical Sciences and LPMC, Nankai University, Tianjin
300071, People's Republic of China}
\email{2120250059@mail.nankai.edu.cn}

\address{(Wei Niu$^{1,2}$) $^1$School of Mathematics and Statistics, Xi'an Jiaotong University, 28 Xianning
West Road, Beilin District, Xi'an, Shaanxi, China}
\address{$^2$Department of Computer Science and Information Theory,
Budapest University of Technology and Economics,
M\H{u}egyetem rakpart 3,
H-1111 Budapest,
Hungary}
\email{wei.niu@stu.xjtu.edu.cn}

\address{(Xiamiao Zhao) Department of Mathematical Sciences, Tsinghua University, Beijing, People's Republic of China}
\email{zxm23@mails.tsinghua.edu.cn}


\begin{abstract}
Let $m$ be a positive integer and $\mathbb{Z}_m$ the residue class ring modulo $m$. The Ruzsa number $R_m$ is defined to be the least integer $r$ such that there is a subset $\mathcal{A}$ of $\mathbb{Z}_m$ satisfying 
$
1\le \sigma_{\mathcal{A}}(n)\le r
$
for any $n\in \mathbb{Z}_m$, where 
$$
\sigma_{\mathcal{A}}(n)
=\#\big\{(a,a')\in\mathcal A^2:
 a+a'\equiv n\pmod{m}\big\}.
$$
Motivated by a 2024 conjecture of Ding and Zhao, we prove 
$
| R_{m+1}-R_m|\le 144.
$
Let $\mathcal{A}$ be a subset of $\mathbb{Z}_m$ satisfying 
$
1\le \sigma_{\mathcal{A}}(n)\le R_m
$
for any $n\in \mathbb{Z}_m$. We also give nontrivial bounds for the size of $\mathcal{A}$. 
Additionally, we provide exact values of $R_m$ for all $m\le 100$, which substantially extends the table of values given by S\'andor and Yang in 2017. Finally, we pose several related problems and prove some partial results.
\end{abstract}
\maketitle

\section{Introduction}
Let $m$ be a positive integer and $\mathbb{Z}_m$ the residue class ring modulo $m$. For any 
$n\in \mathbb{Z}_m$ and $\mathcal{A}\subset \mathbb{Z}_m$, let
$
\sigma_{\mathcal{A}}(n)
=\#\big\{(a,a')\in\mathcal A^2:
 a+a'\equiv n\pmod{m}\big\}.
$
The Ruzsa number $R_m$ is defined to be the least integer $r$ such that there is a subset $\mathcal{A}$ of 
$\mathbb{Z}_m$ satisfying 
$
1\le \sigma_{\mathcal{A}}(n)\le r ~ (\forall~n\in \mathbb{Z}_m).
$
Motivated by the Erd\H os-Tur\'an conjecture on additive bases \cite{ET41}, Ruzsa \cite{Ru90} proved that there is an absolute constant $C>0$ such that $R_m\le C$ for all positive integers $m$. Later, Tang and Chen \cite{TC06} showed that $R_m\le 768$ provided that $m$ is sufficiently large. In a subsequent article \cite{TC07}, Tang and Chen proved that $R_m\le 5120$ for all positive integers $m$. In \cite{Ch08} Chen improved the upper bound, and proved $R_m\le 288$ for all positive integers $m$.
The proof of Chen's bound $R_m\le 288$ consists of two ingredients. First, using Ruzsa's constructions with refinements he proved $R_{2p^2}\le 48$, where $p$ is a prime number. Second, he established the following comparison lemma.

\begin{proposition}[Chen's comparison lemma]\label{prop:1}
Let $m_1$ and $m_2$ be two positive integers with $m_1<m_2<\frac{3}{2}m_1$. Then we have $R_{m_1}\le 6R_{m_2}$.
\end{proposition}

Using $R_{2p^2}\le 48$ and his comparison lemma, Chen obtained the bound $R_m\le 288$ by the explicit prime number theorem. Motivated by Chen's comparison lemma, recently Ding and Zhao \cite{DZ24} proved the following refined comparison lemma.
\begin{proposition}[Ding-Zhao's comparison lemma]\label{prop:2}
Let $m_1$ and $m_2$ be two positive integers with $\frac{3}{2}m_1\le m_2<2m_1$. Then we have $R_{m_2}\le 4R_{m_1}$.
\end{proposition}

As a corollary, Ding and Zhao \cite{DZ24} further improved Chen's bound which asserts 
\begin{align}\label{eq-intro-1}
   R_m\le 192 
\end{align}
for all $m$. On the other hand, Chen \cite{Ch08} pointed out $R_m\ge 3$ if $m\neq 1,2,3$. Interestingly, S\'andor and Yang \cite{SY17} claimed the lower bound  $R_m\ge 6$ for $m\ge 36$. In addition, S\'andor and Yang gave a list of values of $R_m$ up to $m\le 35$.  Based on their numerical values, Ding and Zhao \cite{DZ24} made the following conjecture.

\begin{conjecture}[Ding-Zhao]\label{conj:1}
For any positive integer $m$, we have $\big|R_{m+1}-R_m\big|\le 1$.
\end{conjecture}

However, small inaccuracies crept into the calculations of S\'andor and Yang.  It can be computed that $R_{36}=6$ while $R_{37}=4$, thus giving a counterexample to 
Conjecture \ref{conj:1} for $m=36$. Nevertheless, the lower bound of S\'andor and Yang has no fundamental issues. 
We compute the exact values of $R_m$ up to $36\le m\le 100$, and find that $R_{37}=4$ and $R_{39}=5$ are the only two exceptions. Our computation is certificate-based. For $m\le 45$, lower bounds are verified by finite certificate verification; for $m>45$, Proposition \ref{proposition-Sandor-Yang}
gives $R_m\ge 6$, while the displayed sets give $R_m\le 6$.
We will give an extended list of the values of $R_m$ in the next section, and furthermore formulate the corrected eventual picture. Based on the numerical calculations,  we are confident of the following revised conjecture.

\begin{conjecture}[Ding-Zhao, revised]\label{conj:2}
For all sufficiently large $m$, we have 
$$
\big|R_{m+1}-R_m\big|\le 1.
$$
\end{conjecture}

Probably, Conjecture \ref{conj:1} is true except for $m=36$ and $37$.
One notes easily that the algebraic structures of $\mathbb{Z}_{m}$ and $\mathbb{Z}_{m+1}$ could be significantly different due to the differences between the number of prime factors of $m$ and $m+1$.
Conjecture \ref{conj:2} if true would be a little surprising since it provides an alternative example where combinatorial randomness eventually defeats the rigidity of algebraic structure.
At present, we are not able to attack Conjecture \ref{conj:2}. Clearly, the lower bound $R_m\ge 6$ and the upper bound $R_m\le 192$ indicate that for all sufficiently large $m$ we have
$$
\big|R_{m+1}-R_m\big|\le 192-6=186.
$$
In this note, we give an improvement of the above somewhat trivial bound.

\begin{theorem}\label{thm:1}
    For any positive integer $m$ we have $\big|R_{m+1}-R_m\big|\le 144$.
\end{theorem}

For any positive integer $m$, let $\mathscr{A}_m$ be the set of all subsets $\mathcal{A}\subset [0, m-1]$ such that for any $n\in \mathbb{Z}_m$ we have
$
1\le \sigma_{\mathcal{A}}(n)\le R_m.
$
In \cite[Conjecture 3.4]{DZ24}, the authors asked for estimates for the cardinalities of the sets $\mathcal{A}\in \mathscr{A}_m$. Let $\mathcal{A}\in \mathscr{A}_m$ be a given set. Note that
\begin{align}\label{eq-sec-1-1}
 |\mathcal{A}|^2=\sum_{n\in \mathbb{Z}_m}\sigma_{\mathcal{A}}(n).   
\end{align}
Then, from (\ref{eq-sec-1-1}) and S\'andor-Yang \cite[Lemmas 2.2]{SY17} we know
$$
\Big(\sqrt{2m}-\frac{1}{2}\Big)^2\le |\mathcal{A}|^2\le mR_m.
$$

Our second theorem gives nontrivial bounds on the size of $|\mathcal{A}|$. 

\begin{theorem}\label{thm:2}
For any fixed $\varepsilon>0$, let $m$ be sufficiently large and $\mathcal{A}\in \mathscr{A}_m$ a given set. Then we have
    \begin{align*}
        \Bigg(\floor{\frac{R_m}{2}}+\frac{1}{2}-&\sqrt{\floor{\frac{R_m}{2}}^2-3\floor{\frac{R_m}{2}}+\frac{1}{4}}-\varepsilon\Bigg) m \\
        &\le  |\mathcal{A}|^2 \le \Bigg(\floor{\frac{R_m}{2}}+\frac{1}{2}+\sqrt{\floor{\frac{R_m}{2}}^2-3\floor{\frac{R_m}{2}}+\frac{1}{4}}+\varepsilon\Bigg) m.
    \end{align*}

\end{theorem}

Since $m$ is sufficiently large, $R_m\ge 6$ by Proposition \ref{proposition-Sandor-Yang}. So, we have
$$
\floor{\frac{R_m}{2}}^2-3\floor{\frac{R_m}{2}}+\frac{1}{4}>0.
$$

\begin{corollary}\label{corollary-new}
Let $m$ be sufficiently large and $\mathcal{A}\in \mathscr{A}_m$ a given set. Then we have
$$
|\mathcal{A}|\le \sqrt{191m}.
$$
\end{corollary}
\begin{proof}
Since $R_m\le 192$, we will get the upper bound by making $\varepsilon$ sufficiently small in Theorem \ref{thm:2}.
\end{proof}

\begin{corollary}\label{corollary-1}
For any fixed $\varepsilon>0$, let $m$ be sufficiently large and $\mathcal{A}\in \mathscr{A}_m$ a given set. If $R_m=6$, then  
  $$
\big(\sqrt{3}-\varepsilon\big)\sqrt{m}\le |\mathcal{A}|\le (2+\varepsilon)\sqrt{m}. 
$$  
\end{corollary}
We will see later that the assumption $R_m=6$ in Corollary \ref{corollary-1} is natural due
 to Conjecture \ref{value}.

\section*{Acknowledgments and AI Disclosure}

Yuchen Ding, Wei Niu and Xiamiao Zhao acknowledge support from the CSC programs.
Huixi Li's research is supported by the National Natural Science Foundation of China (Grant No. 12561001). 

OpenAI Codex was used as an auxiliary research and proofreading tool. 
The authors take full responsibility for all proofs and claims in this article.

The data and code used for the computations are available upon request. 

\section{Extended values of \texorpdfstring{$R_m$}{Rm}}

\subsection{The lower bound of S\'andor and Yang revisited}
S\'andor and Yang \cite{SY17} used the mean-square lower bound of
Lev and S\'ark\"ozy \cite{LevSarkozy} together with several finite
reductions to prove that the case \(R_m\le 5\) cannot persist for large
\(m\). Their argument actually gave the following proposition.

\begin{proposition}[S\'andor-Yang]\label{proposition-Sandor-Yang}
If \(m>45\), then \(R_m\ge 6\).
\end{proposition}

Before the proof of Proposition \ref{proposition-Sandor-Yang}, we first introduce the inequality of Lev and S\'ark\"ozy \cite{LevSarkozy}.

\begin{lemma}[Lev-S\'ark\"ozy]\label{lem-L-S}
If $A$ is a subset of a finite
nontrivial abelian group $G$, then for any real number $c$ we have
$$
\sum_{g\in G}\big(r_A(g)-c\big)^2\ge \frac{1}{|G|-1}\left(\frac{|A|^4}{|G|}-2|A|^3+|A|^2|G|\right),
$$
where $r_A(g)=\#\big\{(a_1,a_2)\in A^2: g=a_1+a_2\big\}$.
\end{lemma}

\begin{proof}[Proof of Proposition \ref{proposition-Sandor-Yang}]
We sketch the proof of S\'andor and Yang \cite[Theorem 1.2]{SY17}.
Suppose, for contradiction, that \(R_m\le 5\).  Then there is a set
\(\mathcal{A}\subseteq \Z_m\), \(\card{\mathcal{A}}=k\), such that
\[
1\le \sigma_{\mathcal{A}}(n)\le 5
\]
for any $n\in \Z_m$.
Lev and S\'ark\"ozy's lower bound says that for every real \(c\),
\[
\sum_{n=0}^{m-1}\big(\sigma_{\mathcal{A}}(n)-c\big)^2
\ge
\frac{1}{m-1}\left(\frac{k^4}{m}-2k^3+k^2m\right).
\]
Taking \(c=3\), and then rearranging the right-hand side above gives
\begin{equation}\label{eq:ls-c3}
\sum_{n=0}^{m-1}\big(\sigma_{\mathcal{A}}(n)-3\big)^2
\ge
\frac{k^2(m-k)^2}{m(m-1)}.
\end{equation}
On the other hand, if \(\sigma_{\mathcal{A}}(n)\) is odd, then one of the representations
of \(n\) is diagonal, say \(n=2a\) with \(a\in \mathcal{A}\), since the
off-diagonal pairs occur in symmetric pairs. Since $|\mathcal{A}|=k$, the number of
residue classes with odd \(\sigma_{\mathcal{A}}(n)\) is at most \(k\).  Recall that
\(1\le \sigma_{\mathcal{A}}(n)\le 5\). It then follows that
\[
\big(\sigma_{\mathcal{A}}(n)-3\big)^2\le
\begin{cases}
4,& \text{if } \sigma_{\mathcal{A}}(n)\text{ is odd},\\
1,& \text{if }  \sigma_{\mathcal{A}}(n)\text{ is even}.
\end{cases}
\]
Therefore, by assuming that there are $\ell\le k$ odd $\sigma_{\mathcal{A}}(n)$ we get
\begin{equation}\label{eq:upper-c3}
\sum_{n=0}^{m-1}\big(\sigma_{\mathcal{A}}(n)-3\big)^2\le 4\ell+ (m-\ell) =m+3\ell\le m+3k.
\end{equation}
Combining \eqref{eq:ls-c3} and \eqref{eq:upper-c3} yields the necessary
condition
\begin{equation}\label{eq:first-necessary}
k^2(m-k)^2\le (m+3k)m(m-1).
\end{equation}

Additionally, from S\'andor and Yang \cite[Lemmas 2.2 and 2.3]{SY17} we also have the following elementary bounds
\begin{equation}\label{eq:k-range}
\sqrt{2m}-\frac12<k\le \sqrt{5m}.
\end{equation}
For \(m>91\), the inequalities \eqref{eq:first-necessary} and
\eqref{eq:k-range} are incompatible.  This already proves
\(R_m\ge 6\) for \(m>91\).

For $1\le i\le 5$, let
\[
k_i=\#\big\{n\in \Z_m:\sigma_{\mathcal{A}}(n)=i\big\}.
\]
Since these \(k_i\)'s partition the \(m\) residue classes, any
counterexample must satisfy
\begin{align}
k_1+k_2+k_3+k_4+k_5&=m, \label{eq:dist1}\\
k_1+2k_2+3k_3+4k_4+5k_5&=k^2, \label{eq:dist2}\\
k_1+k_3+k_5&\le k, \label{eq:dist3}
\end{align}
with equality in \eqref{eq:dist3} when \(m\) is odd.  Taking
\(c=k^2/m\) in the Lev--S\'ark\"ozy bound gives the additional
necessary condition
\begin{equation}\label{eq:variance-necessary}
\sum_{i=1}^5\left(i-\frac{k^2}{m}\right)^2k_i
\ge
\frac{k^2(m-k)^2}{m(m-1)}.
\end{equation}
Checking the integer solutions of \eqref{eq:k-range},
\eqref{eq:first-necessary}, and \eqref{eq:dist1}--\eqref{eq:variance-necessary}
already reduces the possible pairs \((m,k)\) to a range whose largest
member is \((50,12)\).

The decisive improvement comes from the additive energy identity
\[
\sum_{n=0}^{m-1}\sigma_{\mathcal{A}}(n)^2=\sum_{n=0}^{m-1}\sigma_{\mathcal{A},-\mathcal{A}}(n)^2.
\]
Here \[\sigma_{\mathcal{A},-\mathcal{A}}(t)=\#\big\{(a,a')\in \mathcal{A}^2:a-a'\equiv t\pmod m\big\}.\]  Note that
\(\sigma_{\mathcal{A},-\mathcal{A}}(0)=k\) and
\[
\sum_{t=1}^{m-1}\sigma_{\mathcal{A},-\mathcal{A}}(t)=k^2-k.
\]
 We write
\[
k^2-k=q(m-1)+r,\qquad 0\le r<m-1,
\]
where 
\[
q=\floor{(k^2-k)/(m-1)}
\]
The sum of squares
\[\sum_{t=1}^{m-1}\sigma_{\mathcal{A},-\mathcal{A}}(t)^2\] is minimized when the nonzero
difference counts are as equal as possible.  Hence
\[
\sum_{t=1}^{m-1}\sigma_{\mathcal{A},-\mathcal{A}}(t)^2\ge q^2(m-1)+(2q+1)r.
\]
Consequently every counterexample must satisfy the sharper inequality
\begin{equation}\label{eq:sharp-variance}
\sum_{i=1}^5\left(i-\frac{k^2}{m}\right)^2k_i
\ge
k^2+q^2(m-1)+(2q+1)r-\frac{k^4}{m}.
\end{equation}

A direct finite check of the integer conditions
\eqref{eq:k-range}, \eqref{eq:first-necessary},
\eqref{eq:dist1}--\eqref{eq:dist3}, and \eqref{eq:sharp-variance}
forces \(m\le 45\).
Therefore no counterexample can
exist for \(m>45\), and hence \(R_m\ge 6\) for \(m>45\).
\end{proof}

\subsection{Extended values of \texorpdfstring{$R_m$}{Rm}}
The computation is certificate-based.  For a proposed row
\((m,\mathcal{A},R)\), we recompute all ordered sums \(a+a'\pmod m\), form the vector
\[
\bigl(\sigma_\mathcal{A}(0),\sigma_{\mathcal{A}}(1),\ldots,
\sigma_{\mathcal{A}}(m-1)\bigr),
\]
and check
\[
\min_n \sigma_{\mathcal{A}}(n)\ge 1,\qquad
\max_n \sigma_{\mathcal{A}}(n)\le R.
\]
For \(m>45\), Proposition \ref{proposition-Sandor-Yang} gives \(R_m\ge 6\).  Hence a
certificate satisfying \(\max_n \sigma_{\mathcal{A}}(n)\le 6\) proves \(R_m=6\).
For \(2\le m\le45\), exactness is verified separately: after checking the displayed
upper certificate for \(R_m\), the verification searches for a certificate with bound
\(R_m-1\) and proves that none exists.

The computation was carried out in two stages.  First, for \(2\le m\le45\), we
verified the displayed upper certificates and then proved that no certificate exists
with bound \(R_m-1\).  This exact verification took about five minutes.  Second, for
larger \(m\), we searched for certificates with maximal representation count at most
\(6\), and then checked each certificate by recomputing all ordered sums.  

The computations were carried out on the server \texttt{lthpc-G2232-G2V2}, equipped with two Intel Xeon Platinum 8360Y processors at 2.40 GHz (72 physical cores, 144 hardware threads) and 1.0 TiB of RAM.
The programs were compiled with \texttt{g++ 11.4.0} using
\texttt{-O3 -std=c++17 -pthread}, and run under Ubuntu 22.04 with Linux kernel 6.8.0-40-generic. No GPU acceleration was used.

The full search for the table took about \(69\) hours in total, including unsuccessful
first passes and later deeper searches.  For each tested cardinality \(k\), we used
three increasing budgets: first \(5000\) attempts, allowing at most \(10^5\) changes in
each attempt; then \(50000\) attempts, allowing at most \(3\cdot10^5\) changes in each
attempt; and finally \(500000\) attempts, allowing at most \(10^6\) changes in each
attempt.  The first pass did not find certificates for \(m=94,98,100\); the
certificates displayed below for these three moduli were found after deeper searches.
The displayed sets are used only as certificates for the claimed upper bounds.  In
particular, for a fixed \(m\), the table records one successful basis; it is not meant
to solve the separate optimization problem of minimizing \(\card{\mathcal{A}}\).

\medskip
\noindent
\begin{minipage}{\textwidth}
\hrule
\smallskip
\textbf{Procedure: certificate search and exact verification.}

\textbf{Input:} a modulus \(m\), a target bound \(R\) for exact verification when
\(m\le45\), the target bound \(6\) for larger moduli, a range of admissible set sizes,
and optional seed sets from nearby moduli.

\begin{enumerate}
\item Normalize candidate sets by translation so that \(0\in\mathcal{A}\), and compute
the full ordered representation vector \(\sigma_\mathcal{A}\).
\item If \(m\le45\), first verify the displayed upper certificate for \(R_m\).  Then try
to build a set with all representation numbers at most \(R_m-1\).  The construction
is exhaustive: at each step one chooses an uncovered residue class and adds a pair
which could represent it.  A partial construction is discarded as soon as some
representation number becomes larger than \(R_m-1\), or when the elementary size
bound \(\lfloor\sqrt{(R_m-1)m}\rfloor\) is exceeded.  If all possibilities are
discarded, no such set exists, and the displayed value of \(R_m\) is exact.
\item If \(m>45\), use Proposition \ref{proposition-Sandor-Yang} for the lower bound
\(R_m\ge6\), and search for a certificate with \(\max_n\sigma_\mathcal{A}(n)\le6\).
The search tries set sizes in the range suggested by the elementary constraints,
using successful sets for nearby moduli as initial candidates.
\item For a candidate \(\mathcal{A}\), measure its defect by first counting uncovered
residue classes and then counting the excess over the bound \(6\).  Thus covering all
residue classes is treated as the first priority, and reducing large representation
numbers is the second priority.
\item Improve the candidate by replacing one element \(a\in\mathcal{A}\) with one
value \(b\notin\mathcal{A}\).  The choice is biased toward removing elements that
contribute to large representation numbers and adding elements that cover missing
residue classes.  Most changes are accepted only when they improve the candidate, but
early in the search some non-improving changes are also allowed.
\item Repeat until either exact nonexistence has been proved, a valid certificate has
been found, or the chosen budget has been exhausted.  Every displayed
certificate is finally verified by recomputing all ordered sums.
\end{enumerate}
\smallskip
\hrule
\end{minipage}
\medskip

Using this procedure, we give below a table of verified values of
\(R_m\) in the range \(2\le m\le100\).  For each listed \(m\), we also provide an
admissible subset \(\mathcal{A}\subseteq\mathbb{Z}_m\) attaining the displayed value.
It is surprising that the values of \(R_m\) seem to converge at \(6\).
\medskip
\begingroup
\small
\setlength{\tabcolsep}{4pt}
\renewcommand{\arraystretch}{1.1}
\setlength{\LTleft}{0pt plus 1fill}
\setlength{\LTright}{0pt plus 1fill}
\begin{longtable}{c|c|>{\raggedright\arraybackslash}p{0.75\textwidth}|c}
\hline
\(m\) & \(\card{\mathcal{A}}\) & \multicolumn{1}{c|}{\(\mathcal{A}\)}  & \(R_m\) \\
\hline
\endfirsthead
\hline
\(m\) & \(\card{\mathcal{A}}\) & \multicolumn{1}{c|}{\(\mathcal{A}\)}  & \(R_m\) \\
\hline
\endhead
\hline
\endfoot
\hline
\endlastfoot
2 & 2 & \{0, 1\} & 2\\
3 & 2 & \{0, 1\} & 2\\
4 & 3 & \{0, 1, 2\} & 3\\
5 & 3 & \{0, 1, 2\} & 3\\
6 & 4 & \{0, 1, 2, 3\} & 4\\
7 & 4 & \{0, 1, 2, 4\} & 3\\
8 & 4 & \{0, 1, 2, 5\} & 4\\
9 & 4 & \{0, 1, 3, 4\} & 4\\
10 & 5 & \{0, 1, 3, 4, 5\} & 4\\
11 & 5 & \{0, 1, 7, 8, 9\} & 4\\
12 & 6 & \{0, 1, 2, 3, 7, 10\} & 4\\
13 & 6 & \{0, 2, 6, 8, 11, 12\} & 4\\
14 & 6 & \{0, 2, 5, 10, 11, 12\} & 4\\
15 & 6 & \{0, 1, 2, 4, 10, 12\} & 4\\
16 & 7 & \{0, 1, 2, 3, 6, 10, 12\} & 5\\
17 & 7 & \{0, 1, 2, 4, 8, 12, 14\} & 5\\
18 & 7 & \{0, 1, 2, 3, 5, 8, 12\} & 5\\
19 & 7 & \{0, 1, 3, 4, 9, 13, 15\} & 4\\
20 & 7 & \{0, 1, 2, 5, 6, 13, 16\} & 5\\
21 & 8 & \{0, 1, 2, 3, 6, 10, 15, 20\} & 5\\
22 & 8 & \{0, 1, 2, 5, 8, 9, 14, 20\} & 5\\
23 & 8 & \{0, 1, 2, 4, 8, 13, 17, 18\} & 5\\
24 & 8 & \{0, 1, 2, 6, 9, 10, 12, 17\} & 5\\
25 & 8 & \{0, 1, 2, 4, 9, 12, 20, 22\} & 5\\
26 & 8 & \{0, 1, 2, 6, 8, 9, 13, 23\} & 6\\
27 & 9 & \{0, 1, 6, 7, 10, 14, 22, 23, 25\} & 5\\
28 & 9 & \{0, 2, 3, 7, 10, 11, 12, 22, 24\} & 5\\
29 & 9 & \{0, 1, 2, 3, 4, 6, 10, 17, 22\} & 6\\
30 & 9 & \{0, 1, 2, 3, 7, 8, 17, 21, 26\} & 6\\
31 & 9 & \{0, 1, 2, 3, 7, 17, 21, 23, 26\} & 6\\
32 & 10 & \{0, 1, 2, 3, 4, 5, 8, 15, 20, 26\} & 6\\
33 & 9 & \{0, 1, 2, 4, 7, 11, 19, 24, 25\} & 6\\
34 & 10 & \{0, 1, 2, 3, 4, 6, 13, 19, 26, 29\} & 6\\
35 & 10 & \{0, 1, 2, 5, 6, 11, 13, 17, 20, 27\} & 5\\
36 & 10 & \{0, 1, 2, 3, 5, 10, 16, 25, 29, 30\} & 6\\
37 & 10 & \{0, 1, 3, 7, 17, 24, 25, 28, 29, 35\} & 4\\
38 & 10 & \{0, 1, 2, 3, 7, 10, 18, 21, 27, 32\} & 6\\
39 & 11 & \{0, 1, 2, 3, 5, 9, 13, 16, 22, 27, 32\} & 5\\
40 & 11 & \{0, 1, 2, 3, 4, 6, 7, 14, 23, 31, 36\} & 6\\
41 & 11 & \{0, 1, 2, 3, 4, 8, 9, 19, 24, 28, 31\} & 6\\
42 & 11 & \{0, 1, 2, 3, 5, 15, 23, 28, 31, 32, 38\} & 6\\
43 & 11 & \{0, 1, 2, 3, 5, 9, 14, 20, 21, 30, 35\} & 6\\
44 & 11 & \{0, 1, 2, 8, 11, 15, 18, 23, 27, 38, 43\} & 6\\
45 & 11 & \{0, 1, 2, 4, 7, 9, 13, 21, 22, 32, 36\} & 6\\
46 & 13 & \{0, 2, 7, 16, 17, 18, 21, 28, 31, 34, 36, 39, 40\} & 6\\
47 & 13 & \{0, 6, 10, 12, 15, 17, 22, 23, 25, 26, 38, 39, 43\} & 6\\
48 & 13 & \{0, 1, 12, 19, 21, 26, 27, 32, 36, 38, 44, 45, 46\} & 6\\
49 & 14 & \{0, 3, 11, 13, 17, 18, 21, 23, 24, 27, 39, 40, 41, 47\} & 6\\
50 & 14 & \{0, 3, 8, 9, 10, 12, 14, 20, 23, 25, 27, 39, 41, 49\} & 6\\
51 & 14 & \{0, 6, 7, 9, 10, 12, 20, 21, 25, 27, 29, 30, 35, 47\} & 6\\
52 & 14 & \{0, 4, 12, 15, 19, 20, 25, 29, 30, 32, 41, 43, 49, 50\} & 6\\
53 & 14 & \{0, 4, 6, 7, 9, 11, 17, 23, 25, 32, 33, 36, 44, 48\} & 6\\
54 & 14 & \{0, 2, 5, 14, 16, 19, 26, 27, 31, 37, 46, 50, 51, 52\} & 6\\
55 & 14 & \{0, 2, 10, 15, 17, 22, 23, 25, 26, 35, 38, 39, 43, 49\} & 6\\
56 & 14 & \{0, 2, 10, 22, 25, 31, 35, 36, 38, 41, 43, 49, 53, 54\} & 6\\
57 & 15 & \{0, 7, 9, 10, 15, 21, 25, 26, 33, 34, 44, 47, 48, 51, 53\} & 6\\
58 & 15 & \{0, 8, 11, 18, 22, 25, 30, 31, 34, 35, 37, 41, 42, 43, 56\} & 6\\
59 & 15 & \{0, 13, 14, 23, 24, 26, 30, 31, 36, 39, 40, 42, 44, 50, 51\} & 6\\
60 & 15 & \{0, 12, 15, 17, 23, 28, 31, 32, 34, 38, 39, 41, 43, 53, 59\} & 6\\
61 & 15 & \{0, 6, 10, 12, 16, 25, 27, 28, 30, 34, 35, 37, 41, 42, 45\} & 6\\
62 & 15 & \{0, 2, 3, 8, 15, 17, 18, 25, 29, 38, 40, 46, 52, 57, 61\} & 6\\
63 & 15 & \{0, 1, 2, 5, 7, 8, 10, 16, 20, 34, 37, 46, 47, 48, 59\} & 6\\
64 & 15 & \{0, 9, 11, 17, 25, 26, 27, 29, 30, 36, 40, 42, 47, 48, 49\} & 6\\
65 & 16 & \{0, 3, 5, 9, 16, 19, 21, 22, 24, 28, 30, 31, 38, 50, 63, 64\} & 6\\
66 & 16 & \{0, 2, 5, 6, 13, 18, 19, 23, 39, 41, 47, 49, 50, 54, 56, 65\} & 6\\
67 & 16 & \{0, 4, 7, 11, 12, 23, 25, 26, 28, 29, 34, 36, 38, 44, 54, 58\} & 6\\
68 & 16 & \{0, 3, 4, 15, 18, 19, 20, 21, 28, 29, 32, 41, 52, 60, 62, 67\} & 6\\
69 & 16 & \{0, 1, 10, 15, 17, 18, 23, 25, 27, 34, 38, 39, 43, 49, 57, 67\} & 6\\
70 & 16 & \{0, 1, 10, 26, 30, 37, 38, 39, 43, 45, 49, 56, 59, 61, 62, 66\} & 6\\
71 & 16 & \{0, 1, 17, 18, 21, 23, 27, 34, 38, 41, 43, 49, 52, 57, 59, 62\} & 6\\
72 & 16 & \{0, 1, 3, 4, 9, 11, 17, 20, 32, 41, 47, 50, 52, 54, 59, 69\} & 6\\
73 & 16 & \{0, 3, 13, 17, 19, 28, 29, 40, 47, 48, 50, 53, 55, 62, 68, 72\} & 6\\
74 & 17 & \{0, 6, 10, 13, 15, 22, 23, 25, 38, 39, 43, 49, 53, 54, 57, 60, 62\} & 6\\
75 & 17 & \{0, 1, 3, 7, 8, 10, 16, 18, 24, 34, 45, 51, 54, 56, 59, 63, 70\} & 6\\
76 & 17 & \{0, 9, 12, 26, 28, 31, 38, 39, 41, 42, 43, 48, 49, 61, 63, 67, 72\} & 6\\
77 & 17 & \{0, 1, 3, 4, 8, 10, 12, 20, 22, 38, 49, 52, 53, 56, 58, 62, 64\} & 6\\
78 & 17 & \{0, 3, 9, 12, 14, 17, 20, 22, 23, 27, 30, 42, 49, 61, 65, 66, 67\} & 6\\
79 & 17 & \{0, 3, 12, 13, 17, 19, 20, 27, 28, 30, 32, 42, 48, 53, 68, 70, 74\} & 6\\
80 & 17 & \{0, 13, 14, 17, 19, 23, 28, 29, 32, 40, 48, 50, 55, 62, 73, 75, 79\} & 6\\
81 & 17 & \{0, 1, 4, 5, 16, 25, 26, 31, 33, 44, 46, 49, 60, 68, 72, 74, 80\} & 6\\
82 & 17 & \{0, 1, 7, 11, 12, 23, 24, 26, 28, 31, 39, 40, 46, 49, 53, 63, 74\} & 6\\
83 & 18 & \{0, 1, 3, 7, 8, 9, 11, 12, 25, 29, 31, 41, 45, 52, 60, 65, 67, 77\} & 6\\
84 & 18 & \{0, 2, 12, 15, 17, 23, 25, 26, 33, 38, 39, 42, 47, 53, 67, 68, 72, 79\} & 6\\
85 & 18 & \{0, 2, 3, 6, 8, 15, 18, 20, 44, 49, 54, 55, 62, 63, 65, 68, 72, 76\} & 6\\
86 & 18 & \{0, 2, 11, 17, 27, 28, 31, 33, 40, 41, 46, 48, 50, 52, 61, 78, 83, 85\} & 6\\
87 & 18 & \{0, 1, 6, 8, 17, 19, 20, 24, 27, 42, 45, 49, 55, 58, 59, 64, 75, 79\} & 6\\
88 & 18 & \{0, 2, 9, 17, 29, 30, 46, 49, 50, 52, 56, 64, 74, 75, 80, 81, 83, 85\} & 6\\
89 & 18 & \{0, 6, 7, 10, 20, 27, 28, 47, 48, 52, 55, 60, 64, 66, 70, 71, 86, 87\} & 6\\
90 & 18 & \{0, 12, 18, 21, 27, 29, 31, 32, 34, 46, 49, 52, 53, 57, 65, 76, 81, 82\} & 6\\
91 & 18 & \{0, 1, 5, 17, 22, 30, 31, 36, 46, 48, 51, 58, 64, 73, 81, 83, 84, 85\} & 6\\
92 & 18 & \{0, 10, 17, 28, 32, 38, 40, 43, 45, 47, 63, 68, 69, 71, 72, 82, 83, 88\} & 6\\
93 & 19 & \{0, 4, 5, 7, 11, 12, 25, 28, 34, 35, 36, 38, 44, 49, 50, 53, 70, 75, 85\} & 6\\
94 & 18 & \{0, 20, 22, 27, 30, 33, 37, 38, 45, 55, 59, 61, 64, 65, 78, 80, 85, 87\} & 6\\
95 & 19 & \{0, 3, 9, 10, 20, 22, 23, 25, 33, 37, 38, 42, 49, 54, 61, 63, 67, 69, 94\} & 6\\
96 & 18 & \{0, 8, 15, 17, 35, 37, 38, 39, 42, 62, 66, 71, 72, 78, 83, 85, 93, 94\} & 6\\
97 & 19 & \{0, 10, 13, 14, 21, 23, 27, 30, 35, 38, 39, 45, 50, 56, 58, 76, 88, 91, 92\} & 6\\
98 & 19 & \{0, 1, 21, 25, 28, 30, 43, 46, 51, 60, 66, 70, 79, 80, 82, 85, 91, 92, 93\} & 6\\
99 & 19 & \{0, 3, 10, 12, 30, 32, 38, 42, 45, 47, 61, 63, 68, 69, 74, 82, 83, 85, 86\} & 6\\
100 & 19 & \{0, 4, 5, 9, 12, 28, 31, 33, 39, 42, 46, 48, 58, 59, 60, 67, 83, 90, 99\} & 6\\
\end{longtable}
\endgroup
Based on the table list of $R_m$ above, it seems reasonable  to make the following conjecture whose positive answers would give solutions to basic problems involving Ruzsa's numbers asked by Chen \cite[Problem 1 and Problem 2]{Ch08}.

\begin{conjecture}\label{value}
    $R_m=6$ for $m\ge 40$.
\end{conjecture}

\section{Proof of Theorem \ref{thm:1}}
 Motivated by Proposition \ref{prop:1} and Proposition \ref{prop:2}, we will prove the following new kind of comparison lemma. It is worth mentioning that our new lemma is different from Proposition \ref{prop:2} because of the constraints of $m_1$ and $m_2$ therein. 

 \begin{lemma}\label{lem:1}
   For any positive integer $m$ we have both
   $$R_m\le 4R_{m+1} \quad \text{and} \quad R_{m+1}\le 4R_{m}.$$
 \end{lemma}
 \begin{proof}
 {\bf We first prove the assertion $R_{m+1}\le 4R_{m}$}. 

 Suppose that $\mathcal{A}$ is a subset of $\mathbb{Z}_m$ such that for any $n\in \mathbb{Z}_m$ we have
 $$
 1\le \sigma_{\mathcal{A}}(n)\le R_{m}.
 $$
Without loss of generality, we may assume $0\le a\le m-1$ for any $a\in \mathcal{A}$. Define
$$
\mathcal{A}_{1}=\big\{a\in \mathcal{A}: a\ge \floor{m/2}\big\}
$$
and
$$
\mathcal{A}_{+}=\mathcal{A}_{1}+1:=\{a+1: a\in \mathcal{A}_{1}\}.
$$
Now, we let
$
\mathcal{B}=\mathcal{A}\cup \mathcal{A}_{+},
$
and regard $\mathcal{B}$ as a subset of $\mathbb{Z}_{m+1}$.

First, we show that $\mathcal{B}$ is an additive basis of $\mathbb{Z}_{m+1}$, i.e., $\sigma_{\mathcal{B}}(n)\ge 1$ for any $n\in \mathbb{Z}_{m+1}$. For any $0\le n\le m-1$, there are two elements $a, a'\in \mathcal{A}$ such that
$$
n\equiv a+a'\pmod{m},
$$
which implies $n=a+a'$ or $n=a+a'-m$ since $0\le a+a'\le 2m-2$. If $n=a+a'$, then 
$
n\equiv a+a'\pmod{m+1}.
$
Now, suppose that $n=a+a'-m$. Then $a+a'\ge m$, and hence at least one of $a$ and $a'$ is $\ge \floor{m/2}$. We may assume without loss of generality that $a\ge \floor{m/2}$. Then we have
$$
n=(a+1)+a'-(m+1)\equiv (a+1)+a'\pmod{m+1},
$$
which means that $\sigma_{\mathcal{B}}(n)\ge 1$ for any $0\le n\le m-1$. It remains to show $\sigma_{\mathcal{B}}(m)\ge 1$. Note that
there are two elements $a_1$ and $a_2$ of $\mathcal{A}$ such that
$$
m-1\equiv a_1+a_2\pmod{m}.
$$
Since $0\le a_1+a_2\le 2m-2$, we must have $m-1=a_1+a_2$. It follows that at least one of $a_1$ and $a_2$ is $\ge  \floor{m/2}$. Without loss of generality, we assume that $a_1\ge  \floor{m/2}$. Hence,
$$
m=(a_1+1)+a_2 \equiv (a_1+1)+a_2\pmod{m+1},
$$
completing the proof of $\sigma_{\mathcal{B}}(n)\ge 1$ for any $n\in \mathbb{Z}_{m+1}$.

Next, we will prove $\sigma_{\mathcal{B}}(n)\le 4R_m$ for any $0\le n\le m$. Let
$$
\mathcal{A}_{2}=\mathcal{A}\setminus\mathcal{A}_1=\big\{a\in \mathcal{A}: a< \floor{m/2}\big\}.
$$
By this symbol, we have
$$
\mathcal{B}=\mathcal{A}_1\cup \mathcal{A}_{2}\cup\mathcal{A}_{+}.
$$
For any subsets $\mathcal{S}_1$ and $\mathcal{S}_2$ of integers, define
$$
f_{\mathcal{S}_1,\mathcal{S}_2}(n)=\#\big\{(s_1,s_2)\in \mathcal{S}_1\times\mathcal{S}_2: n=s_1+s_2\big\}.
$$
We first consider the range $0\le n\le m-1$. Suppose that
\begin{align*}
    n\equiv b_1+b_2\pmod{m+1},
\end{align*}
where $b_1$ and $b_2$ are two elements of $\mathcal{B}$.
Then $n=b_1+b_2$ or $n=b_1+b_2-(m+1)$ since $0\le b_1+b_2\le 2m$. In the case $n=b_1+b_2$ we have
\begin{align*}
    \#\big\{(b_1, b_2)&\in \mathcal{B}^2 :  n=b_1 +b_2\big\}
    \le f_{\mathcal{A}_{1},\mathcal{A}_{1}}(n)+f_{\mathcal{A}_{1},\mathcal{A}_{2}}(n)+f_{\mathcal{A}_{2},\mathcal{A}_{1}}(n)\\
&+f_{\mathcal{A}_{1},\mathcal{A}_{+}}(n)+f_{\mathcal{A}_{+},\mathcal{A}_{1}}(n)+f_{\mathcal{A}_{2},\mathcal{A}_{2}}(n)+f_{\mathcal{A}_{2},\mathcal{A}_{+}}(n)+f_{\mathcal{A}_{+},\mathcal{A}_{2}}(n)+f_{\mathcal{A}_{+},\mathcal{A}_{+}}(n).
\end{align*}
Note that
$$
f_{\mathcal{A}_{1},\mathcal{A}_{+}}(n)=f_{\mathcal{A}_{1},\mathcal{A}_{1}}(n-1)=0,\quad f_{\mathcal{A}_{2},\mathcal{A}_{+}}(n)=f_{\mathcal{A}_{2},\mathcal{A}_{1}}(n-1),
$$
$$
f_{\mathcal{A}_{+},\mathcal{A}_{1}}(n)=f_{\mathcal{A}_{1},\mathcal{A}_{1}}(n-1)=0,\quad f_{\mathcal{A}_{+},\mathcal{A}_{2}}(n)=f_{\mathcal{A}_{1},\mathcal{A}_{2}}(n-1),
$$
and
$$
f_{\mathcal{A}_{+},\mathcal{A}_{+}}(n)=f_{\mathcal{A}_{1},\mathcal{A}_{1}}(n-2)=0
$$
since $a+a'\ge 2\floor{m/2}\ge m-1>n-1$ for any $a, a'\in \mathcal{A}_{1}$,
from which we obtain
\begin{align}\label{eq-proof-1}
    \#\big\{(b_1, b_2)\in \mathcal{B}^2& :  n=b_1 +b_2\big\}
    \le f_{\mathcal{A}_{1},\mathcal{A}_{1}}(n)+f_{\mathcal{A}_{1},\mathcal{A}_{2}}(n)\nonumber\\
    &+f_{\mathcal{A}_{2},\mathcal{A}_{1}}(n)+f_{\mathcal{A}_{2},\mathcal{A}_{2}}(n)+f_{\mathcal{A}_{2},\mathcal{A}_{1}}
    (n-1)+f_{\mathcal{A}_{1},\mathcal{A}_{2}}
    (n-1).
\end{align}
Similarly, in the case $n=b_1+b_2-(m+1)$ we have
\begin{align}\label{eq-proof-2}
    \#\big\{(b_1,b_2&)\in \mathcal{B}^2: n + (m+1)= b_1+b_2\big\}\nonumber\\
    &\le f_{\mathcal{A}_{1},\mathcal{A}_{1}}(n+1+m)+f_{\mathcal{A}_{1},\mathcal{A}_{2}}(n+1+m)
    +f_{\mathcal{A}_{2},\mathcal{A}_{1}}(n+1+m)\nonumber\\
    &\quad \quad +f_{\mathcal{A}_{1},\mathcal{A}_{+}}(n+1+m)+f_{\mathcal{A}_{+},\mathcal{A}_{1}}(n+1+m)+f_{\mathcal{A}_{2},\mathcal{A}_{2}}(n+1+m)\nonumber\\
    &\quad \quad +f_{\mathcal{A}_{2},\mathcal{A}_{+}}(n+1+m)+f_{\mathcal{A}_{+},\mathcal{A}_{2}}(n+1+m)+f_{\mathcal{A}_{+},\mathcal{A}_{+}}(n+1+m)\nonumber\\
    &= f_{\mathcal{A}_{1},\mathcal{A}_{1}}(n+1+m)+f_{\mathcal{A}_{1},\mathcal{A}_{2}}(n+1+m)
    +f_{\mathcal{A}_{2},\mathcal{A}_{1}}(n+m+1)\nonumber\\
    &\quad \quad +2f_{\mathcal{A}_{1},\mathcal{A}_{1}}(n+m)+f_{\mathcal{A}_{2},\mathcal{A}_{2}}(n+1+m)\nonumber\\
    &\quad \quad +f_{\mathcal{A}_{2},\mathcal{A}_{1}}(n+m)+f_{\mathcal{A}_{1},\mathcal{A}_{2}}(n+m)+f_{\mathcal{A}_{1},\mathcal{A}_{1}}(n-1+m).
\end{align}
Hence, we conclude from (\ref{eq-proof-1}) and (\ref{eq-proof-2}),
\begin{align}\label{eq-proof-3}
    \sigma_{\mathcal{B}}(n)&\le f_{\mathcal{A},\mathcal{A}}(n)+2f_{\mathcal{A},\mathcal{A}}(n+m)+f_{\mathcal{A},\mathcal{A}}(n-1)+f_{\mathcal{A},\mathcal{A}}(n-1+m)+f_{\mathcal{A},\mathcal{A}}(n+1+m)\nonumber\\
    &\le 2\#\big\{(a_1,a_2)\in \mathcal{A}^2: n\equiv a_1+a_2 \pmod{m}\big\}\nonumber\\
    &\quad \quad +\#\big\{(a_1,a_2)\in \mathcal{A}^2: n-1\equiv a_1+a_2\pmod{m}\big\}\nonumber\\
    &\quad \quad +\#\big\{(a_1,a_2)\in \mathcal{A}^2: n+1\equiv a_1+a_2\pmod{m}\big\}\nonumber\\
    &\le 4R_m.
\end{align}
It remains to estimate $\sigma_{\mathcal B}(m)$. Suppose that
$$
m\equiv b_1+b_2\pmod{m+1},
$$
where $b_1,b_2\in\mathcal B$. Since $0\le b_1+b_2\le 2m$, we must have $b_1+b_2=m$. Each $b_i$ can be written in the form $b_i=a_i+\varepsilon_i$, where $a_i\in\mathcal A$ and $\varepsilon_i\in\{0,1\}$; here $\varepsilon_i=1$ can occur only when $a_i\in\mathcal A_1$. Counting all such choices gives an upper bound, and hence
\begin{align*}
\sigma_{\mathcal B}(m)
&\le f_{\mathcal A,\mathcal A}(m)+2f_{\mathcal A,\mathcal A}(m-1)+f_{\mathcal A,\mathcal A}(m-2)\\
&\le \sigma_{\mathcal A}(0)+2\sigma_{\mathcal A}(-1)+\sigma_{\mathcal A}(-2)\\
&\le 4R_m.
\end{align*}
Therefore, we have $R_{m+1}\le 4 R_m$.

\medskip

{\bf Proof of the assertion $R_{m}\le 4R_{m+1}$}. 

The cases $m=1$ and $2$ are trivial. Now, we assume $m\ge 3$.
Suppose that $\mathcal{A}$ is a subset of $\mathbb{Z}_{m+1}$ such that for any $n\in \mathbb{Z}_{m+1}$ we have
 $$
 1\le \sigma_{\mathcal{A}}(n)\le R_{m+1}.
 $$
Without loss of generality, we may assume $0\le a\le m$ for any $a\in \mathcal{A}$. Define
$$
\mathcal{A}_{1}=\big\{a\in \mathcal{A}: a\ge m/2+1\big\}
$$
and
$$
\mathcal{A}_{-}=\mathcal{A}_{1}-1:=\{a-1: a\in \mathcal{A}_{1}\}.
$$
Now, we let
$
\mathcal{B}=\mathcal{A}\cup \mathcal{A}_{-},
$
and regard $\mathcal{B}$ as a subset of $\mathbb{Z}_{m}$.

First, we show that $\mathcal{B}$ is an additive basis of $\mathbb{Z}_{m}$, i.e., $\sigma_{\mathcal{B}}(n)\ge 1$ for any $n\in \mathbb{Z}_{m}$. We first consider $n=0$. Since $\mathcal A$ is an additive basis of $\mathbb Z_{m+1}$, there exist $a,a'\in\mathcal A$ such that
$$
m\equiv a+a'\pmod{m+1}.
$$
Since $0\le a+a'\le 2m$, we must have $a+a'=m$. Hence $0\equiv a+a'\pmod m$, and so $\sigma_{\mathcal B}(0)\ge1$.

Now let $1\le n\le m-1$. There are two elements $a, a'\in \mathcal{A}$ such that
$$
n\equiv a+a'\pmod{m+1},
$$
which implies $n=a+a'$ or $n=a+a'-(m+1)$ since $0\le a+a'\le 2m$. If $n=a+a'$, then 
$
n\equiv a+a'\pmod{m}.
$
Now, suppose that $n=a+a'-(m+1)$. Then $a+a'=n+m+1\ge m+2$. If both $a$ and $a'$ were $<m/2+1$, then, since $a$ and $a'$ are integers, we would have $a+a'\le m+1$, a contradiction. Thus at least one of $a$ and $a'$ is $\ge m/2+1$. We may assume without loss of generality that $a\ge m/2+1$. Then we have
$$
n=(a-1)+a'-m\equiv (a-1)+a'\pmod{m},
$$
which means that $\sigma_{\mathcal{B}}(n)\ge 1$ for any $1\le n\le m-1$. Hence $\mathcal B$ is an additive basis of $\mathbb Z_m$.

Next, we will prove $\sigma_{\mathcal{B}}(n)\le 4R_{m+1}$ for any $0\le n\le m-1$. Let
$$
\mathcal{A}_{2}=\mathcal{A}\setminus\mathcal{A}_1=\big\{a\in \mathcal{A}: a< m/2+1\big\}.
$$
By this symbol, we have
$$
\mathcal{B}=\mathcal{A}_1\cup \mathcal{A}_{2}\cup\mathcal{A}_{-}.
$$
It would be convenient to consider $\sigma_{\mathcal{B}}(0)$ separately. Suppose that
$$
b_1+b_2\equiv 0\pmod{m},
$$
where $b_1,b_2\in\mathcal B$. Then $b_1+b_2=0$ or $m$ or $2m$. The cases $b_1+b_2=0$ and $b_1+b_2=2m$ contribute at most two pairs in total. If $b_1+b_2=m$, then the corresponding elements of $\mathcal A$ have sum $m$, $m+1$, or $m+2$, according as neither, exactly one, or both of the two summands come from $\mathcal A_-$. Therefore
$$
\sigma_{\mathcal{B}}(0)\le \sigma_{\mathcal{A}}(-1)+\sigma_{\mathcal{A}}(0)+\sigma_{\mathcal{A}}(1)+2\le 3R_{m+1}+2\le 4R_{m+1},
$$
where we used $R_{m+1}\ge 2$.

In what follows, let $n$ be an integer satisfying $1\le n\le m-1$. Suppose that
\begin{align*}
    n\equiv b_1+b_2\pmod{m},
\end{align*}
where $b_1$ and $b_2$ are two elements of $\mathcal{B}$.
Then $n=b_1+b_2$ or $n=b_1+b_2-m$ since $0\le b_1+b_2\le 2m$. In the case $n=b_1+b_2$ we have
\begin{align*}
    \#\big\{(b_1, b_2)&\in \mathcal{B}^2 :  n=b_1 +b_2\big\}
    \le f_{\mathcal{A}_{1},\mathcal{A}_{1}}(n)+f_{\mathcal{A}_{1},\mathcal{A}_{2}}(n)+f_{\mathcal{A}_{2},\mathcal{A}_{1}}(n)\\
&+f_{\mathcal{A}_{1},\mathcal{A}_{-}}(n)+f_{\mathcal{A}_{-},\mathcal{A}_{1}}(n)+f_{\mathcal{A}_{2},\mathcal{A}_{2}}(n)+f_{\mathcal{A}_{2},\mathcal{A}_{-}}(n)+f_{\mathcal{A}_{-},\mathcal{A}_{2}}(n)+f_{\mathcal{A}_{-},\mathcal{A}_{-}}(n).
\end{align*}
Note that
$$
f_{\mathcal{A}_{1},\mathcal{A}_{-}}(n)=f_{\mathcal{A}_{1},\mathcal{A}_{1}}(n+1)=0,\quad f_{\mathcal{A}_{2},\mathcal{A}_{-}}(n)=f_{\mathcal{A}_{2},\mathcal{A}_{1}}(n+1),
$$
$$
f_{\mathcal{A}_{-},\mathcal{A}_{1}}(n)=f_{\mathcal{A}_{1},\mathcal{A}_{1}}(n+1)=0,\quad f_{\mathcal{A}_{-},\mathcal{A}_{2}}(n)=f_{\mathcal{A}_{1},\mathcal{A}_{2}}(n+1),
$$
and
$$
f_{\mathcal{A}_{-},\mathcal{A}_{-}}(n)=f_{\mathcal{A}_{1},\mathcal{A}_{1}}(n+2)=0
$$
since $a+a'\ge m+2>n+2$ for any $a, a'\in \mathcal{A}_{1}$,
from which we obtain
\begin{align}\label{eq-proof-4}
    \#\big\{(b_1, b_2)\in \mathcal{B}^2& :  n=b_1 +b_2\big\}
    \le f_{\mathcal{A}_{1},\mathcal{A}_{1}}(n)+f_{\mathcal{A}_{1},\mathcal{A}_{2}}(n)\nonumber\\
    &+f_{\mathcal{A}_{2},\mathcal{A}_{1}}(n)+f_{\mathcal{A}_{2},\mathcal{A}_{2}}(n)+f_{\mathcal{A}_{2},\mathcal{A}_{1}}
    (n+1)+f_{\mathcal{A}_{1},\mathcal{A}_{2}}
    (n+1).
\end{align}
Similarly, in the case $n=b_1+b_2-m$ we have
\begin{align}\label{eq-proof-5}
    \#\big\{(b_1,b_2&)\in \mathcal{B}^2: n +m= b_1+b_2\big\}\nonumber\\
    &\le f_{\mathcal{A}_{1},\mathcal{A}_{1}}(n+m)+f_{\mathcal{A}_{1},\mathcal{A}_{2}}(n+m)
    +f_{\mathcal{A}_{2},\mathcal{A}_{1}}(n+m)\nonumber\\
    &\quad \quad +f_{\mathcal{A}_{1},\mathcal{A}_{-}}(n+m)+f_{\mathcal{A}_{-},\mathcal{A}_{1}}(n+m)+f_{\mathcal{A}_{2},\mathcal{A}_{2}}(n+m)\nonumber\\
    &\quad \quad +f_{\mathcal{A}_{2},\mathcal{A}_{-}}(n+m)+f_{\mathcal{A}_{-},\mathcal{A}_{2}}(n+m)+f_{\mathcal{A}_{-},\mathcal{A}_{-}}(n+m)\nonumber\\
    &= f_{\mathcal{A}_{1},\mathcal{A}_{1}}(n+m)+f_{\mathcal{A}_{1},\mathcal{A}_{2}}(n+m)
    +f_{\mathcal{A}_{2},\mathcal{A}_{1}}(n+m)\nonumber\\
    &\quad \quad +2f_{\mathcal{A}_{1},\mathcal{A}_{1}}(n+m+1)+f_{\mathcal{A}_{2},\mathcal{A}_{2}}(n+m)\nonumber\\
    &\quad \quad +f_{\mathcal{A}_{2},\mathcal{A}_{1}}(n+m+1)+f_{\mathcal{A}_{1},\mathcal{A}_{2}}(n+m+1)+f_{\mathcal{A}_{1},\mathcal{A}_{1}}(n+2+m).
\end{align}
Hence, we conclude from (\ref{eq-proof-4}) and (\ref{eq-proof-5}),
\begin{align*}
    \sigma_{\mathcal{B}}(n)&\le f_{\mathcal{A},\mathcal{A}}(n)+2f_{\mathcal{A},\mathcal{A}}(n+1+m)+f_{\mathcal{A},\mathcal{A}}(n+1)+f_{\mathcal{A},\mathcal{A}}(n+2+m)+f_{\mathcal{A},\mathcal{A}}(n+m)\nonumber\\
    &\le2\#\big\{(a_1,a_2)\in \mathcal{A}^2: n\equiv a_1+a_2 \pmod{m+1}\big\}\nonumber\\
    &\quad \quad +\#\big\{(a_1,a_2)\in \mathcal{A}^2: n+1\equiv a_1+a_2\pmod{m+1}\big\}\nonumber\\
    &\quad \quad +\#\big\{(a_1,a_2)\in \mathcal{A}^2: n-1\equiv a_1+a_2\pmod{m+1}\big\}\nonumber\\
    &\le 4R_{m+1}.
\end{align*}
Therefore, we have $R_{m}\le 4 R_{m+1}$, completing the proof of our lemma.
 \end{proof}

Now, we turn to the proof of Theorem \ref{thm:1}.

\begin{proof}[Proof of Theorem \ref{thm:1}]

    We separate the argument into two cases.

    {\it Case I.} $R_m\le R_{m+1}$. In this case, by Lemma \ref{lem:1} we have
    \begin{align*}
        \big| R_m-R_{m+1}\big|=R_{m+1}-R_m\le R_{m+1}-\frac{1}{4}R_{m+1}=\frac{3}{4}R_{m+1}.
    \end{align*}
    Using the upper bound (\ref{eq-intro-1}), we have
    $$
    \big| R_m-R_{m+1}\big|\le \frac{3}{4}\cdot 192=144.
    $$

    {\it Case II.} $R_{m+1}\le R_{m}$. In this case, again by Lemma \ref{lem:1} we have
    \begin{align*}
        \big| R_m-R_{m+1}\big|=R_{m}-R_{m+1}\le R_{m}-\frac{1}{4}R_{m}=\frac{3}{4}R_{m}\le \frac{3}{4}\cdot 192=144,
    \end{align*}
    completing the proof of Theorem \ref{thm:1}.
\end{proof}

\section{Proof of Theorem \ref{thm:2}}
Following the argument of S\'andor and Yang \cite[Theorem 1.2]{SY17}, we will make use of the Lev-S\'ark\"ozy inequality.

\begin{proof}[Proof of Theorem \ref{thm:2}]
    Suppose that $|\mathcal{A}|=k$. Let $\alpha=k^2/m$. Then for any $c$,
    \begin{align}\label{eq-3-1}
        \sum_{n\in \mathbb{Z}_m}\big(\sigma_{\mathcal{A}}(n)-c\big)^2&=\sum_{n\in \mathbb{Z}_m}\big(\sigma_{\mathcal{A}}(n)-\alpha+\alpha-c\big)^2\nonumber\\
        &=\sum_{n\in \mathbb{Z}_m}\big(\sigma_{\mathcal{A}}(n)-\alpha\big)^2+2(\alpha-c)\sum_{n\in\mathbb{Z}_m}\big(\sigma_{\mathcal{A}}(n)-\alpha\big)+m(\alpha-c)^2\nonumber\\
        &=\sum_{n\in \mathbb{Z}_m}\big(\sigma_{\mathcal{A}}(n)-\alpha\big)^2+m(\alpha-c)^2
    \end{align}
    By Lemma \ref{lem-L-S} and (\ref{eq-3-1}) we have
    \begin{align}\label{eq-3-2}
     \sum_{n\in \mathbb{Z}_m}\big(\sigma_{\mathcal{A}}(n)-c\big)^2 \ge \frac{k^2(m-k)^2}{m(m-1)}+m\left(\frac{k^2}{m}-c\right)^2.
    \end{align}
    
    On the other hand, let $s=\floor{R_m/2}$ and $ c=s+1$. Since $R_m\le 192$, we have $s\le 96$; the error terms below are therefore uniform in $s$. Then we have 
    \begin{align*}
     \sum_{n\in \mathbb{Z}_m}\big(\sigma_{\mathcal{A}}(n)-c\big)^2& =\sum_{\substack{n\in \mathbb{Z}_m\\ 2| \sigma_{\mathcal{A}}(n)}}\big(\sigma_{\mathcal{A}}(n)-c\big)^2+\sum_{\substack{n\in \mathbb{Z}_m\\ 2\nmid  \sigma_{\mathcal{A}}(n)}}\big(\sigma_{\mathcal{A}}(n)-c\big)^2\\
     &\le (s-1)^2\sum_{\substack{n\in \mathbb{Z}_m\\ 2| \sigma_{\mathcal{A}}(n)}}1+s^2\sum_{\substack{n\in \mathbb{Z}_m\\ 2\nmid  \sigma_{\mathcal{A}}(n)}}1.
    \end{align*}
    Note that the number of odd values of $\sigma_{\mathcal{A}}(n)$ is at most $k$ since $|\mathcal{A}|=k$. Hence, we get
    \begin{align}\label{eq-3-3}
\sum_{n\in \mathbb{Z}_m}\big(\sigma_{\mathcal{A}}(n)-c\big)^2\le ks^2+(m-k)(s-1)^2=m(s-1)^2+k(2s-1).
    \end{align}
  Combining (\ref{eq-3-2}) with (\ref{eq-3-3}) we obtain  
  \begin{align*}
\frac{k^2(m-k)^2}{m(m-1)}+m\left(\frac{k^2}{m}-c\right)^2 \le m(s-1)^2+k(2s-1).
    \end{align*}
    Dividing both sides by $m$, we get
    \begin{align}\label{eq-new-1}
    \alpha\cdot\frac{(m-k)^2}{m(m-1)}+\left(\alpha-c\right)^2\le (s-1)^2+\frac{k}{m}(2s-1).
    \end{align}
Since 
$$
k^2=|\mathcal{A}|^2\le mR_m\le 192m, 
$$
we know $k\le \sqrt{192m}$, and hence 
$k/m\to 0$ as $m\to\infty$. 
     Then 
$$
\frac{(m-k)^2}{m(m-1)}=1+o(1).
$$    
From (\ref{eq-new-1}) and above inequalities we obtain  
    \begin{align}\label{eq-3-4}
\alpha+(\alpha-s-1)^2\le (s-1)^2+o(1).
    \end{align}
    Expanding (\ref{eq-3-4}) and rearranging it gives
    \begin{align*}
      \alpha^2-(2s+1)\alpha+4s\le o(1),  
    \end{align*}
from which we deduce 
\begin{align*}
   s+\frac{1}{2}-\sqrt{\left(s+\frac{1}{2}\right)^2-4s}-\varepsilon \le   \alpha  \le s+\frac{1}{2}+\sqrt{\left(s+\frac{1}{2}\right)^2-4s}+\varepsilon,
    \end{align*}
    where $\varepsilon$ is arbitrarily small providing that $m$ is sufficiently large (in terms of $\varepsilon$).
\end{proof}

\section{Related problems and results}
We conclude the paper by posing several unsolved problems involving Ruzsa's numbers. Recall that $\mathscr{A}_m$ denotes the set of all subsets $\mathcal{A}\subset [0,m-1]$ such that for any $n\in \mathbb{Z}_m$ we have
$$
1\le \sigma_{\mathcal{A}}(n)\le R_m.
$$

\begin{problem}\label{probl:1}
Let $m$ be a sufficiently large integer. Is it true that there are two subsets $\mathcal{A}_1, \mathcal{A}_2\in \mathscr{A}_m$ such that 
\(
|\mathcal{A}_1|\neq |\mathcal{A}_2|?
\)
\end{problem}

For $\mathcal A=\{a_1<a_2<\cdots<a_t\}\subset [0,m-1]$, put $a_{t+1}=a_1+m$ and define the largest cyclic gap of $\mathcal A$ by
\[
G(\mathcal A)=\max_{1\le i\le t}\{a_{i+1}-a_i\}.
\]

\begin{problem}\label{probl:2}
Is it true that 
\[
\lim_{m\to\infty}\sup_{\mathcal A\in\mathscr A_m}\frac{G(\mathcal A)}{m}=0?
\]
\end{problem}

For Problem \ref{probl:2}, we point out the following result.
\begin{proposition}\label{pro-prob-2}
We have
\[
\limsup_{m\to\infty}\sup_{\mathcal A\in\mathscr A_m}\frac{G(\mathcal A)}{m}\le \frac{1}{2}.
\] 
\end{proposition}
\begin{proof}
Assume the contrary, and suppose that 
\[
\limsup_{m\to\infty}\sup_{\mathcal A\in\mathscr A_m}\frac{G(\mathcal A)}{m}=\eta> \frac{1}{2}.
\]
Then there exist arbitrarily large $m$ and some $\mathcal A=\{a_1<a_2<\cdots<a_t\}\in\mathscr A_m$ such that
\[
G(\mathcal A)>\eta^*m,
\]
where
$$
\eta^*=\frac{\eta+1/2}{2}>\frac{1}{2}.
$$
Choose $j$ such that $a_{j+1}-a_j=G(\mathcal A)$, where $a_{t+1}=a_1+m$. Let 
$$
\mathcal{A}^*=\big\{b_i: b_i\equiv a_i-a_j\pmod{m},\ 0\le b_i\le m-1,\ 1\le i\le t\big\}.
$$
Clearly, $\mathcal{A}^*\in \mathscr{A}_m$. By the choice of the largest gap, $\mathcal{A}^*\cap(0,G(\mathcal A))=\emptyset$; since $G(\mathcal A)>\eta^*m$, it follows that $\mathcal{A}^*\cap(0,\eta^*m]=\emptyset$. We claim that
\[
\sigma_{\mathcal A^*}(n)=0
\]
for any integer $n$ satisfying
\[
0<n<\min\big\{(2\eta^*-1)m,\eta^*m\big\}.
\]
Indeed, suppose that $n\equiv x+y\pmod m$ with $x,y\in\mathcal A^*$. If one of $x$ and $y$ is $0$, then the other one would have to lie in $(0,\eta^*m)$, which is impossible. If both $x$ and $y$ are nonzero, then $x,y>\eta^*m$. Hence $x+y>2\eta^*m>m$, and after reduction modulo $m$ the corresponding residue is larger than $(2\eta^*-1)m$. This again is impossible. Thus no such $n$ is represented, contradicting $\mathcal A^*\in\mathscr A_m$. This proves the proposition.
\end{proof}

For even $m$ and $\mathcal A\in\mathscr A_m$, put
\[
p_m(\mathcal A)=\frac{\#\big\{a\in\mathcal{A}: 2\mid a\big\}}{|\mathcal{A}|}.
\]

\begin{problem}\label{probl:3}
Is it true that 
\[
\lim_{\substack{m\to\infty\\ 2\mid m}}\sup_{\mathcal A\in\mathscr A_m}\left|p_m(\mathcal A)-\frac12\right|=0?
\]
\end{problem}

For Problem \ref{probl:3}, we can prove the following result.

\begin{proposition}\label{pro-prob-3}
For every even $m$ and every $\mathcal A\in\mathscr A_m$, let
\[
E_{\mathcal A}=\big\{a\in\mathcal A:2\mid a\big\},\qquad
O_{\mathcal A}=\big\{a\in\mathcal A:2\nmid a\big\}.
\]
Then
\[
\frac m2\le |E_{\mathcal A}||O_{\mathcal A}|
\le \frac{\floor{R_m/2}}{2}m.
\]
Moreover, for any fixed $\varepsilon>0$, if $m$ is sufficiently large and even, then
\begin{align*}
\frac{1}{2}-\sqrt{\frac{1}{4}-\frac{1}{2(R_m^*+\varepsilon)}}
\le p_m(\mathcal A)
\le
\frac{1}{2}+\sqrt{\frac{1}{4}-\frac{1}{2(R_m^*+\varepsilon)}},
\end{align*}
where 
$$
R_m^*=\floor{\frac{R_m}{2}}+\frac{1}{2}+\sqrt{\floor{\frac{R_m}{2}}^2-3\floor{\frac{R_m}{2}}+\frac{1}{4}}.
$$
\end{proposition}
\begin{proof}
Since $m$ is even, every odd residue class can only be written as the sum of an element of $E_{\mathcal A}$ and an element of $O_{\mathcal A}$. Moreover, for any odd residue class, all representations occur in off-diagonal pairs. Since $\mathcal A\in\mathscr A_m$, each odd residue class has at least two representations: if $n=e+o$ with $e\in E_{\mathcal A}$ and $o\in O_{\mathcal A}$, then $(e,o)$ and $(o,e)$ are two distinct ordered representations. Hence
$$
2|E_{\mathcal{A}}||O_{\mathcal{A}}|= \sum_{\substack{n\in \mathbb{Z}_m\\ 2\nmid n}}\sigma_{\mathcal{A}}(n)
\ge 2\cdot \frac{m}{2}=m,
$$
which gives the lower bound
$$
|E_{\mathcal{A}}||O_{\mathcal{A}}|\ge \frac{m}{2}.
$$
On the other hand, each odd residue class has an even number of representations, and this number is at most $R_m$. Hence it is at most $2\floor{R_m/2}$. Therefore
$$
2|E_{\mathcal{A}}||O_{\mathcal{A}}|= \sum_{\substack{n\in \mathbb{Z}_m\\ 2\nmid n}}\sigma_{\mathcal{A}}(n)
\le 2\floor{\frac{R_m}{2}}\cdot \frac m2,
$$
which gives
$$
|E_{\mathcal{A}}||O_{\mathcal{A}}|\le \frac{\floor{R_m/2}}{2}m.
$$

Now suppose that $m$ is sufficiently large and even, and put $p=p_m(\mathcal A)=|E_{\mathcal A}|/|\mathcal A|$. Then $|O_{\mathcal A}|=(1-p)|\mathcal A|$. From the lower bound above we obtain
\begin{align}\label{problem-1}
    p(1-p)|\mathcal{A}|^2\ge \frac{m}{2}.
\end{align}
By Theorem \ref{thm:2} we have 
\begin{align}\label{problem-2}
|\mathcal{A}|^2\le \big(R_m^*+\varepsilon\big)m,
\end{align}
provided that $m$ is large enough in terms of $\varepsilon$. Combining (\ref{problem-1}) with (\ref{problem-2}) gives
$$
p(1-p)\ge \frac{1}{2(R_m^*+\varepsilon)}.
$$
Equivalently,
$$
p^2-p+\frac{1}{2(R_m^*+\varepsilon)}\le 0,
$$
from which the desired inequality follows.
\end{proof}

\begin{corollary}\label{cor-prob-3}
Suppose that $R_m=6$ for all sufficiently large $m$. Then
\[
\limsup_{\substack{m\to\infty\\2\mid m}}\sup_{\mathcal A\in\mathscr A_m}\left|p_m(\mathcal A)-\frac12\right|
\le \sqrt{\frac18}.
\]
\end{corollary}
\begin{proof}
If $R_m=6$, then
\[
R_m^*=3+\frac12+\sqrt{9-9+\frac14}=4.
\]
The result follows from Proposition \ref{pro-prob-3} by letting $\varepsilon\to0$.
\end{proof}

Suppose $\mathcal{A}\in \mathscr{A}_m$. Let $L_{\mathcal{A}}$ denote the largest integer $\ell$ for which there exist $x,d\in\mathbb Z_m$ such that
\[
x,\ x+d,\ \ldots,\ x+(\ell-1)d
\]
are distinct elements of $\mathcal A$. Let
$$
L_m=\max_{\mathcal{A}\in \mathscr{A}_m}L_{\mathcal{A}}.
$$
Assuming Conjecture \ref{value}, we have
$L_m\le 6$ for $m\ge 40$. Generally, $L_m\le R_m$ for any $m$.
In fact, if 
$
x,\ x+d,\ \cdots,\ x+(\ell-1)d \in \mathcal{A}
$
are distinct, then the $\ell$ ordered pairs
\[
\big(x+id,\ x+(\ell-1-i)d\big),\qquad 0\le i\le \ell-1,
\]
all give representations of the same residue class $2x+(\ell-1)d$ modulo $m$. Thus
$$
\sigma_{\mathcal{A}}\big(2x+(\ell-1)d\big)\ge \ell,
$$
which means $\ell\le R_m$.
Naturally, determining the values of $L_m$ would be another open question. It may also be useful to record $L_{\mathcal A}$ for the certificates in the table above. This leads to the following natural problem.
\begin{problem}\label{probl:4}
Is it true that $L_m=R_m$ for infinitely many $m$?
\end{problem}

\begin{problem}\label{probl:5}
    Does $\log |\mathscr{A}_m|$ admit an asymptotic formula? 
    Is it true that
    $$
    |\mathscr{A}_m|\le |\mathscr{A}_{m+1}|
    $$ for all sufficiently large $m$? 
\end{problem}

As a partial result toward the first part of Problem \ref{probl:5}, we prove the following.
\begin{proposition}\label{pro-prob-new-1}
For all sufficiently large integers $m$, we have
\[
    \frac12\log m-\frac12\log 191
    \le
    \log |\mathscr A_m|
    \le
    \sqrt{191m}\log\!\left(e\sqrt{\frac m{191}}\right)+O(\log m).
\]
In particular, we have
\(
    |\mathscr{A}_m|< |\mathscr{A}_{m+1}|
\)
infinitely often.
\end{proposition}

\begin{proof}
Let $\mathcal{A}\in\mathscr A_m$, and write $k=|\mathcal{A}|$.
By Corollary \ref{corollary-new} we have
$$
k\le \sqrt{191m}
$$
for all sufficiently large $m$. Therefore,
\[
    \big|\mathscr A_m\big|
    \le
    \sum_{j\le \sqrt{191m}}\binom mj.
\]
Let $K=\lceil\sqrt{191m}\rceil$. Using the standard estimate
\[
    \sum_{j\le K}\binom mj
    \le
    \left(\frac{em}{K}\right)^K
    \quad (K\le m/2),
\]
we obtain
\[
    \log |\mathscr A_m|
    \le
    K\log\!\left(\frac{em}{K}\right)
    \le
    \sqrt{191m}\log\!\left(e\sqrt{\frac m{191}}\right)+O(\log m).
\]

It remains to give a lower bound. Take any $\mathcal{A}\in\mathscr{A}_m$.
For every $h\in\mathbb Z_m$, the translate
\[
    \mathcal{A}+h=\big\{a+h:a\in \mathcal{A}\big\}
\]
also belongs to $\mathscr A_m$. Let
\[
    \operatorname{Stab}(\mathcal{A})=\{h\in\mathbb Z_m:\mathcal{A}+h=\mathcal{A}\}.
\]
The number of distinct translates of $\mathcal{A}$ is
\[
    \frac{m}{|\operatorname{Stab}(\mathcal{A})|}.
\]
Moreover, for any fixed $a\in \mathcal{A}$ we have
\[
    a+\operatorname{Stab}(\mathcal{A})\subset \mathcal{A},
\]
and hence
\[
    |\operatorname{Stab}(\mathcal{A})|\le |\mathcal{A}|=k.
\]
Consequently,
\[
    |\mathscr A_m|
    \ge
    \frac m k.
\]
Using $k\le \sqrt{191m}$, we get
\[
    |\mathscr A_m|
    \ge
    \sqrt{\frac m{191}}.
\]
Taking logarithms gives
\[
    \log |\mathscr A_m|
    \ge
    \frac12\log m-\frac12\log 191.
\]
Finally, the lower bound shows that $|\mathscr A_m|$ is unbounded. If
\[
|\mathscr A_m|<|\mathscr A_{m+1}|
\]
held only finitely often, then the sequence $|\mathscr A_m|$ would be eventually non-increasing, and hence bounded, a contradiction. This proves the proposition.
\end{proof}

\begin{remark}\label{rem-affine}
The same lower-bound argument can also be applied to affine images. Indeed, for every $u\in(\mathbb Z_m)^\times$ and $h\in\mathbb Z_m$, the set
\[
u\mathcal A+h=\{ua+h:a\in\mathcal A\}
\]
also belongs to $\mathscr A_m$. If
\[
\operatorname{AffStab}(\mathcal A)=\{(u,h)\in(\mathbb Z_m)^\times\times\mathbb Z_m:u\mathcal A+h=\mathcal A\},
\]
then the affine orbit of $\mathcal A$ gives
\[
|\mathscr A_m|\ge \frac{m\varphi(m)}{|\operatorname{AffStab}(\mathcal A)|}.
\]
In particular, if there are a constant $C>0$ and sets $\mathcal A_m\in\mathscr A_m$ such that
\[
|\operatorname{AffStab}(\mathcal A_m)|\le C|\mathcal A_m|^2
\]
for infinitely many $m$, then Corollary \ref{corollary-new} gives, along those $m$,
\[
|\mathscr A_m|\ge \frac{\varphi(m)}{191C}.
\]
This would substantially improve the general lower bound in Proposition \ref{pro-prob-new-1}.
\end{remark}

The following problem is in the same flavor as Problem \ref{probl:5}.

\begin{problem}\label{probl:6}
     Is it true that for infinitely many $m$ we have
    $$
    \min_{\mathcal{A}\in \mathscr{A}_m}|\mathcal{A}|> \min_{\mathcal{A}\in \mathscr{A}_{m+1}}|\mathcal{A}|?
    $$ 
\end{problem}

 \begin{problem}\label{probl:7}
    Let $m$ be sufficiently large and $\mathcal{A}\in \mathscr{A}_m$. Is it true that for any $2\le v\le R_m$ there exists some $n\in \mathbb{Z}_m$ such that $\sigma_{\mathcal{A}}(n)=v$?
\end{problem}

\begin{remark}
The endpoint $v=R_m$ in Problem \ref{probl:7} is automatic. Indeed, if $\sigma_{\mathcal A}(n)<R_m$ for every $n\in\mathbb Z_m$, then $\mathcal A$ would satisfy
\[
1\le \sigma_{\mathcal A}(n)\le R_m-1\qquad(n\in\mathbb Z_m),
\]
contradicting the definition of $R_m$.
\end{remark}

\begin{proposition}\label{pro-value-2}
Let $\mathcal A\in\mathscr A_m$. If no residue class has exactly two representations, then
\[
|\mathcal A|^2\ge 4m-3|\mathcal A|.
\]
Consequently, for any sequence $m_j\to\infty$ and $\mathcal A_j\in\mathscr A_{m_j}$, if
\[
\limsup_{j\to\infty}\frac{|\mathcal A_j|^2}{m_j}<4,
\]
then $2$ is attained by $\sigma_{\mathcal A_j}$ for all sufficiently large $j$.
\end{proposition}

\begin{proof}
Let
\[
\Omega=\{n\in\mathbb Z_m:\sigma_{\mathcal A}(n)\ \text{is odd}\}.
\]
The number of residues in $\Omega$ is at most $|\mathcal A|$, since an odd value of $\sigma_{\mathcal A}(n)$ can occur only if $n=2a$ for some $a\in\mathcal A$. If no residue class has exactly two representations, then every residue outside $\Omega$ has an even positive number of representations different from $2$, and hence has at least $4$ representations. Every residue in $\Omega$ has at least one representation. Therefore
\[
|\mathcal A|^2=\sum_{n\in\mathbb Z_m}\sigma_{\mathcal A}(n)
\ge 4(m-|\Omega|)+|\Omega|=4m-3|\Omega|
\ge 4m-3|\mathcal A|.
\]
The final assertion follows from this inequality and Corollary \ref{corollary-new}, since $|\mathcal A_j|/m_j\to0$.
\end{proof}

	\bibliographystyle{plain}

\end{document}